\documentclass{amsart}

\usepackage{amssymb} \usepackage{amsfonts} \usepackage{amsmath}
\usepackage{amsthm} 
\usepackage{epsfig} 
\usepackage{color}
\usepackage{amscd}
\usepackage[all]{xy}
\usepackage{pslatex}

\newtheorem{lemma}{Lemma}
\newtheorem{teo}[lemma]{Theorem}

\newtheorem{cor}[lemma]{Corollary}

\theoremstyle{definition}

\theoremstyle{remark}

\newcommand{\matCP} {\ensuremath {\mathbb{CP}}}

\author{Bruno Martelli}
\address{Dipartimento di Matematica ``Tonelli'', Largo Pontecorvo 5, 56127 Pisa, Italy}
\email{martelli at dm dot unipi dot it}

\title{A finite set of local moves for Kirby calculus}

\begin{document}

\begin{abstract}
We exhibit a finite set of local moves that connect any two surgery presentations of the same 3-manifold via framed links in $S^3$. The moves are handle-slides and blow-downs/ups of a particular simple kind.
\end{abstract}

\maketitle

A \emph{framed link} $L\subset S^3$ in the three-sphere is a link equipped with a section (the \emph{framing}) of the unitary normal bundle. The framing is considered only up to isotopy and is notoriously determined by assigning an integer to each component of $L$, which counts the algebraic intersection of the framing with the standard longitude.

Let $L\subset S^3$ be a framed link in the three-sphere. A \emph{surgery} along $L$ is a standard cut-and-paste operation in three-dimensional topology that consists of removing from $S^3$ a solid torus neighborhood of each component of $L$ and gluing it back via a different map, hence producing a new closed 3-manifold $N$. The map sends a meridian of the solid torus to a curve parallel to the framing of $L$, and this requirement is enough to determine the new 3-manifold $N$ up to diffeomorphism. 

This cut-and-paste operation has a natural four-dimensional interpretation: we may see surgery as the result of attaching a four-dimensional 2-handle to $S^3$, thus constructing a four-dimensional cobordism from $S^3$ to $N$.

A celebrated theorem of Lickorish-Wallace \cite{Li, Wa} states that every closed orientable 3-manifold $N$ can be obtained by surgerying along some framed link $L$ in the three-sphere. It is easy to construct many different framed links giving rise to the same 3-manifold $N$: there are in fact various types of moves that transform $L$ into a new framed link $L'$ withouth affecting $N$. The most important ones are handles-slides and blow-downs.

A \emph{handle-slide} consists of sliding one 2-handle onto another; this move does not affect the four-dimensional cobordism and hence $N$ is also preserved. The link is modified as follows: pick two components of $L$, select a handlebody containing them as in Fig.~\ref{HS:fig}, and then modify one component as shown in the picture. 

\begin{figure}
 \begin{center}
  \includegraphics[width = 8cm]{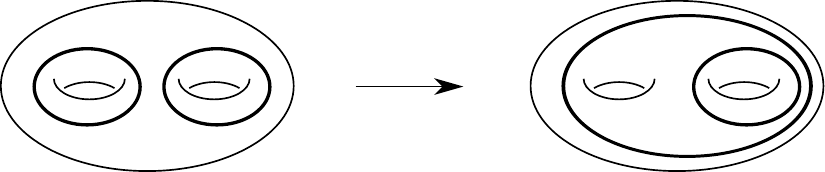}
 \end{center}
 \caption{A handle-slide. The framing of all components is horizontal. Note that the handlebody may be kotted in $S^3$.}
 \label{HS:fig}
\end{figure}

If a component $K$ of $L$ is unknotted and has framing $\pm 1$, it determines a $\pm\matCP^2$-summand in the four-dimensional cobordism: the removal of this factor is a \emph{topological blow-down} which also does not affect $N$. The corresponding move is shown in Fig.~\ref{FR:fig}. The inverse of a blow-down is of course called a \emph{blow-up}.

\begin{figure}
 \begin{center}
  \includegraphics[width = 7cm]{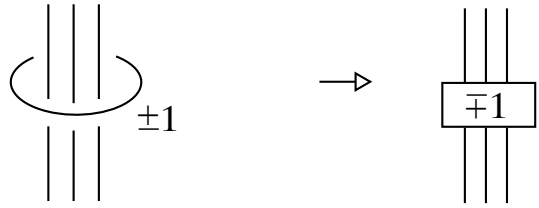}
 \end{center}
 \caption{A topological blow-down. The number $n\geqslant 0$ of vertical strands crossing the unknot is arbitrary (here $n=3$). The box marked with $+1$ ($-1$) indicates one full counterclockwise (clockwise) twist.}
 \label{FR:fig}
\end{figure}

These moves generate all possible moves. More precisely, let $L$ and $L'$ be two framed link representing the same 3-manifold $N$. A theorem of Kirby \cite{Ki} says that $L$ and $L'$ are related by handle-slides and blow-downs/ups with $n=0$ vertical strands. Subsequently, Fenn and Rourke \cite{FR} have shown that blow-downs/ups with arbitrary number $n$ of strands are also enough to connect $L$ and $L'$. 

Recall that a \emph{framed tangle} in a 3-disc $D$ is a collection of disjoint framed properly embedded knots and arcs in $D$. We say that a move is \emph{local} if it consists of substituting a framed tangle in some 3-disc $D$ with another framed tangle in $D$, sharing the same framed endpoints. For instance, all the moves shown in Fig.~\ref{FR:fig} and Fig.~\ref{moves_paper:fig} are local.

We may now say that Kirby's theorem furnishes a finite set of non-local moves, whereas Fenn-Rourke's theorem gives an infinite list of local moves. We exhibit here a finite list of local moves.

\begin{figure}
 \begin{center}
  \includegraphics[width = 12.5 cm]{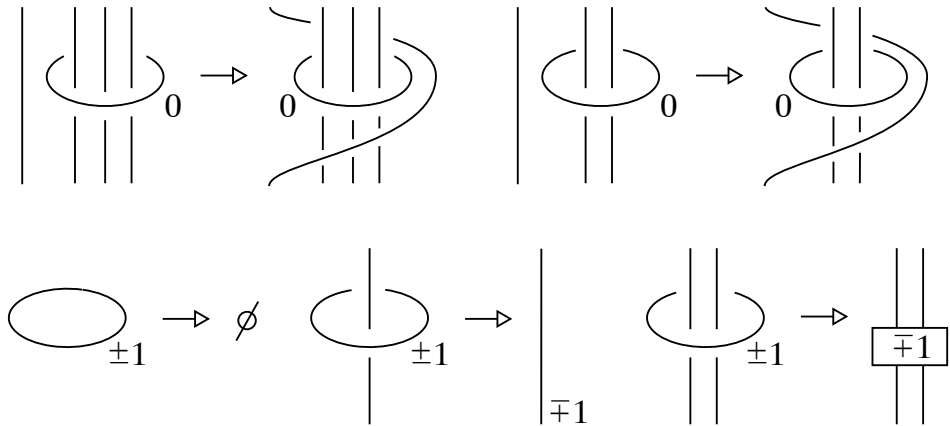}
 \end{center}
 \caption{A finite set of local moves.}
 \label{moves_paper:fig}
\end{figure}

\begin{teo}
Let $L$ and $L'$ be two framed links in $S^3$ representing the same 3-manifold $N$ via surgery. The links are connected by a sequence of moves of the type shown in Fig.~\ref{moves_paper:fig} and their inverses.
\end{teo}
\begin{proof}
We show that our moves generate a topological blow-down with any number $n$ of strands, see Fig.~\ref{FR:fig}. We consider the case $n=3$ with sign $-1$: all other cases are solved analogously. 

\begin{figure}
 \begin{center}
  \includegraphics[width = 12.5 cm]{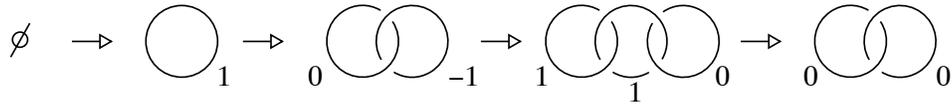}
 \end{center}
 \caption{How to create a 0-framed Hopf link. The 0-framed Hopf link represents the 4-manifold $S^2\times S^2$ and this is the usual sequence that produces $S^2\times S^2$ from $S^3$ via three topological blow-ups and one blow-down. }
 \label{blow:fig}
\end{figure}

We first note that our moves allow to create a 0-framed Hopf link, see Fig.~\ref{blow:fig}. In Fig.~\ref{chain1:fig} we use the same sequence of moves to add a tail of two 0-framed unknots. We perform this operation $n-1$ times in order to get a tail as in Fig.~\ref{chain2:fig}-left. Then we slide the vertical strands as in Fig.~\ref{chain2:fig}.

\begin{figure}
 \begin{center}
  \includegraphics[width = 12 cm]{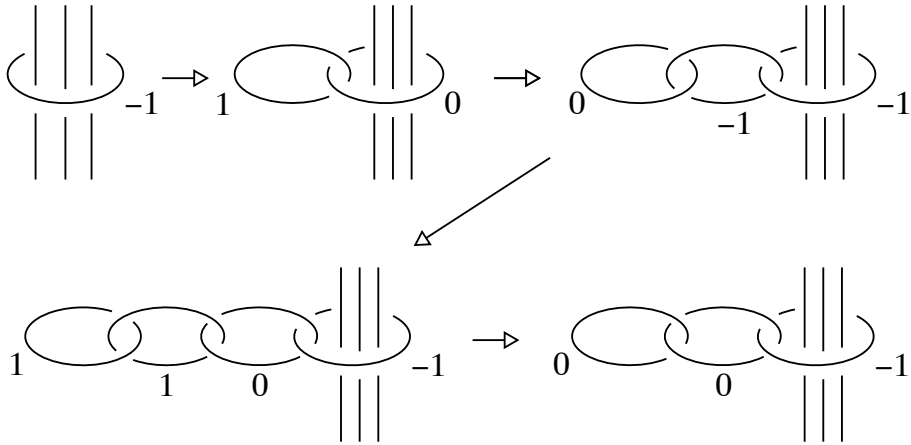}
 \end{center}
 \caption{We add a tail of 0-framed unknots, following Fig.~\ref{blow:fig}.}
 \label{chain1:fig}
\end{figure}

\begin{figure}
 \begin{center}
  \includegraphics[width = 12.5 cm]{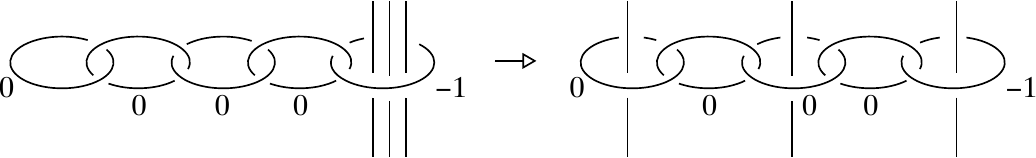}
 \end{center}
 \caption{We slide the vertical strands in order to isolate each one from the others.}
 \label{chain2:fig}
\end{figure}

\begin{figure}
 \begin{center}
  \includegraphics[width = 12.5 cm]{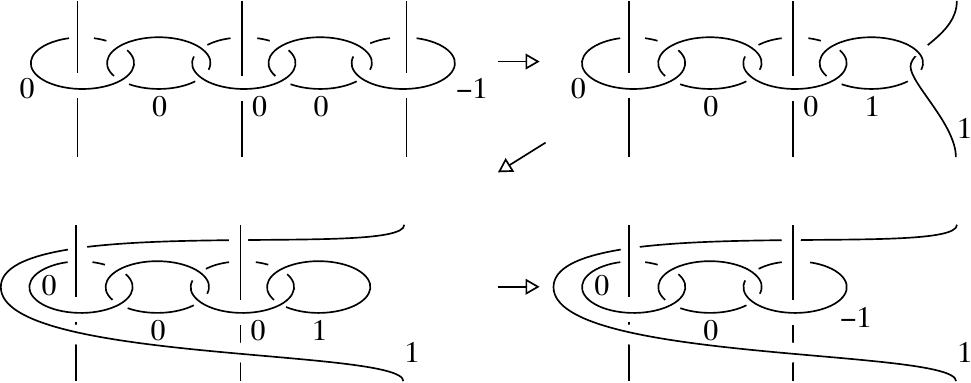}
 \end{center}
 \caption{We slide the rightmost strand. We make a blow-down, two handle-slides, and one blow-down again. }
 \label{chain3:fig}
\end{figure}

\begin{figure}
 \begin{center}
  \includegraphics[width = 12 cm]{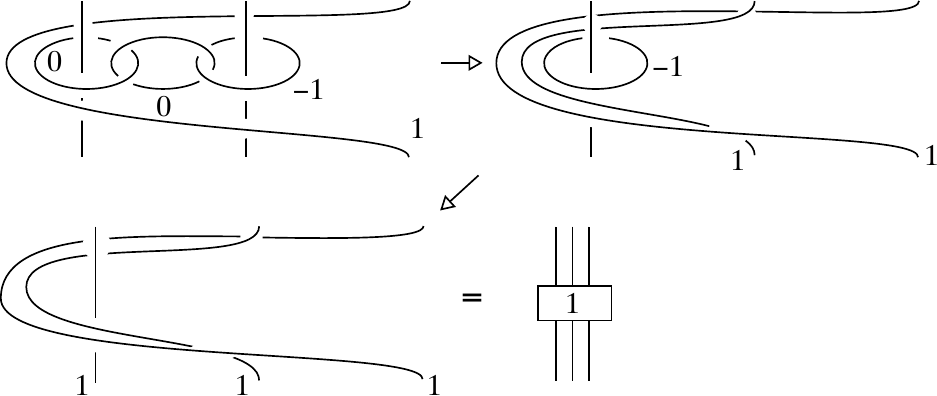}
 \end{center}
 \caption{We conclude the proof by iteration.}
 \label{chain4:fig}
\end{figure}

Having isolated the vertical strands, we slide the rightmost strand to the left as in Fig.~\ref{chain3:fig}. We repeat this operation for each strand as in Fig.~\ref{chain4:fig}. The result is as in Fig.~\ref{FR:fig}-right, as required.
\end{proof}

Of course one may choose a different finite set of generating moves. Blow-ups/downs with at most 5 strands for instance suffice.

\begin{cor}
Let $L$ and $L'$ be two links representing the same 3-manifold $N$ via surgery. The 
links are connected by a sequence of blow-ups/downs with $n\leqslant 5$ strands.
\end{cor}
\begin{proof}
We prove that the handleslides in Fig.~\ref{moves_paper:fig} can be generated by blow-downs/ups with $n\leqslant 5$ strands. This is done in Fig.~\ref{FR2:fig}.
\begin{figure}
 \begin{center}
  \includegraphics[width = 12 cm]{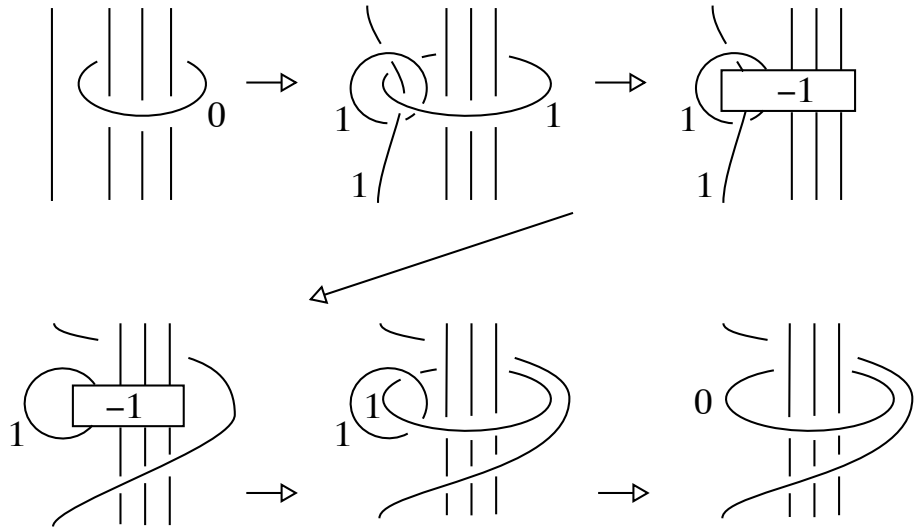}
 \end{center}
 \caption{We make one blow-up, one blow-down, we slide one strand, and make one more blow-up and blow.down.}
 \label{FR2:fig}
\end{figure}
\end{proof}

\thanks{This paper was motivated by a nice question on Mathoverflow \cite{algori}.}

\end{document}